\documentclass[a4paper, 10pt]{article}      

\usepackage[utf8]{inputenc}
\usepackage[T1]{fontenc}
\usepackage[english]{babel}

\usepackage{amsmath,amssymb,amsthm}
\usepackage{mathrsfs}
\usepackage{color}
\usepackage[colorlinks]{hyperref}
\usepackage{graphicx}
\usepackage{ulem}
\usepackage[dvipsnames]{xcolor}
\usepackage{cancel}
\usepackage{dsfont}
\usepackage{mathtools}
\usepackage{stmaryrd}
\usepackage{nicefrac}
\usepackage{comment}     
\usepackage{geometry}                                                  

\newtheorem{thm}{Theorem}

\newtheorem{cor}[thm]{Corollary}

\theoremstyle{definition}
\newtheorem{dfn}[thm]{Definition}
\newtheorem{asm}{Assumption}

\newtheorem{rem}[thm]{Remark}

\newcommand{\N}{\mathbb{N}}

\newcommand{\R}{\mathbb{R}}

\newcommand{\Z}{\mathbb{Z}}

\newcommand{\bB}{\mathbf{B}}

\newcommand{\bS}{\mathbf{S}}

\newcommand{\bX}{\mathbf{X}}

\newcommand{\vx}{\mathbf{x}}

\newcommand{\vf}{\mathbf{f}}

\newcommand{\vp}{\mathbf{p}}

\newcommand{\balpha}{\boldsymbol\alpha}

\newcommand{\bomega}{\boldsymbol\omega}
\newcommand{\bsigma}{\boldsymbol\sigma}
\newcommand{\bphi}{\boldsymbol\phi}
\newcommand{\bpsi}{\boldsymbol\psi}

\newcommand{\mF}{\mathrm{F}}
\newcommand{\mA}{\mathrm{A}}
\newcommand{\mQ}{\mathrm{Q}}

\newcommand{\cO}{\mathcal{O}}
\newcommand{\cL}{\mathcal{L}}

\newcommand{\cQ}{\mathcal{Q}}
\newcommand{\cT}{\mathcal{T}}

\newcommand{\lint}{\llbracket}
\newcommand{\rint}{\rrbracket}

\renewcommand{\epsilon}{\varepsilon}

\renewcommand{\emph}[1]{{\it #1}}

\title{\LARGE \bf
Specialized effective Positivstellensätze for improved convergence rates of the moment-SOS hierarchy
}

\author{Corbinian Schlosser$^{1}$ and Matteo Tacchi$^{2}$
\thanks{The work of Corbinian Schlosser was supported by the European Research Council (grant REAL 947908)}
}

\begin{document}

\maketitle

\footnotetext[1]{Corbinian Schlosser is with Inria Paris Centre, France
        {\tt\small corbinian.schlosser@inria.fr}}%
\footnotetext[2]{Matteo Tacchi is with Univ. Grenoble Alpes, CNRS, Grenoble INP (Institute of Engineering Univ. Grenoble Alpes), GIPSA-lab, France
{\tt\small matteo.tacchi@gipsa-lab.grenoble-inp.fr}}%

\begin{abstract}
    Recently a moment-sum-of-squares hierarchy for exit location estimation of stochastic processes has been presented. When restricting to the special case of the unit ball, we show that the solutions approach the optimal value by a super-polynomial rate. To show this result we state a new effective Positivstellensatz on the sphere with quadratic degree bound based on a recent Positivstellensatz for trigonometric polynomials on the hypercube and pair it with a recent effective Positivstellensatz on the unit ball. At the present example, we aim to highlight the effectiveness of specialized Positivstellensätze for the moment-SoS hierarchy and their interplay with problem intrinsic properties.
\end{abstract}

\section{Introduction}

In the last years, there has been increasing interest in leveraging efficient Positivstellensätze to address various classes of (nonlinear) problems, particularly those formulated within the framework of the generalized moment problem (GMP). The versatility of GMPs finds application across diverse domains, ranging from geometry, where it facilitates volume computation of semialgebraic sets as well as set separation \cite{korda2022urysohn}, \cite{henrion2009volume,lasserre2017stokes,tacchi2022exploiting}, to dynamical systems encompassing optimal control \cite{lasserre2008nonlinear},  stability analysis \cite{korda2020computing,oustry2019inner,korda2014convex,jones2021converse,schlosser2021converging, goluskin2020bounding}, as well as partial differential equations~\cite{marx2020burgers}, and more generally calculus of variations \cite{henrion2023occupation}. Moreover, GMPs play a pivotal role in studying stochastic systems and have recently offered computational insights into exit location estimation \cite{henrion2023moment}, infinite-time averaging \cite{fantuzzi2016bounds}, invariant measures computation \cite{korda2021convex}, and probability of unsafety\cite{miller2023unsafe} to name only a few.

The moment-sum-of-squares (SoS) hierarchy provides a convergent method for solving GMPs, but investigation of convergence rates beyond polynomial optimization remains scarce. Notable instances where explicit convergence rates were explored are \cite{korda2017convergence,korda2018convergence} based on an effective version of Putinar's Positivstellensatz \cite{nie2007complexity}. Often a large disparity persists between the theoretical convergence bounds and the rates observed in practice. Recent improvements \cite{baldi2023effective} on effective versions of Putinar's Positivstellensatz strongly reduced this gap. Focusing on specific semialgebraic sets, such as the unit ball or the hypercube, closed the gap further
 \cite{slot2022sum,laurent2023effective}; and lower bounds have been established \cite{baldi2024degree}.

In previous work \cite{schlosser2024convergence}, a framework for obtaining convergence rates from GMPs is investigated. This text can be seen as a follow-up, in which we aim at transferring the improvement in the effective Putinar's Positivstellensatz for specific semialgebraic sets to GMPs at the example of exit location for stochastic differential equations \cite{henrion2023moment}.  Compared to \cite{schlosser2024convergence}, we strongly improve the convergence rates for the exit location on the unit ball for stochastic processes and point out the interplay between properties of the GMP and the application of an adapted Positivstellensatz.

\section{Notation}
We work with the standard notations for usual sets $\R$ (real numbers), $\Z$ (integers), $\N$ (natural integers). For $a<b \in \N$, $\lint a,b \rint := \{a,a+1\ldots,b-1,b\}$ is the set of all integers between $a$ and $b$. For $\balpha = (\alpha_1,\ldots,\alpha_n) \in \N^n$, $|\balpha| := \alpha_1+\ldots+\alpha_n$ is the range of $\balpha$ and $ (x_1,\ldots,x_n) = \vx \mapsto \vx^{\balpha} := x_1^{\alpha_1}\cdots x_n^{\alpha_n}$ is the corresponding monomial. For $n,d \in \N$, we set $\N^n_d := \{\balpha \in \N^n \; ; |\balpha| \leq d\}$, and $\R_d[\vx]:=\{\vx \mapsto \sum_{|\balpha| \leq d} c_{\balpha} \, \vx^{\balpha} \; ; (c_{\balpha})_{\balpha} \in \R^{\N^n_d}\}$ is the space of polynomials of degree at most $d$, $\R[\vx] := \cup_{d\in\N} \R_d[\vx]$ is the space of polynomials. We denote the euclidean unit ball $\{x \in \R^n : \|x\|_2\leq 1\}$ by $\bB$, by $\mathring{\bB}$ its interior, and by $\bS$ its boundary.

\section{Preliminaries}
One of the central pillars in this text is Positivstellensätze. Among these is the celebrated Putinar's Positivstellensatz \cite[Theorem 1.3 \& Lemma 3.2]{putinar1993positive} and its effective versions. In this text, we will apply two effective Positivstellensätze, one for polynomials and one for trigonometric polynomials.

\subsection{Real algebraic Positivstellensätze}

We call a set $\bX \subset \R^n$ semialgebraic, if there exist $m \in \N$ and polynomials $p_1,\ldots,p_m \in \R[\vx]$ such that
\begin{equation}\label{eq:def:SemialgebraicSet}
    \bX = \{\vx \in \R[\vx] : p_1(\vx)\geq 0,\ldots,p_m(\vx)\geq 0\}.
\end{equation}
If a set $\bX$ is given by (\ref{eq:def:SemialgebraicSet}), we denote it by $\bX(\vp)$.

The moment-SoS hierarchy links polynomial optimization problems with characterizing non-negative polynomials: For $f\in \R[\vx]$ and a compact semialgebraic set $\bX$, it holds
\begin{equation*}
    \begin{tabular}{ccccc}
         $\min\limits_{\vx}$ & $f(x)$ & $=$ & $\max\limits_{\lambda}$ & $\lambda$ \\
         s.t. & $\vx \in \bX$ & & s.t & $f-\lambda \geq 0 \text{ on } \bX$
    \end{tabular}
\end{equation*}
and the optimization problem reduces to certification of non-negativity, see to \cite{lasserre2015introduction,parrilo2000structured} for detailed surveys on polynomial optimization. A natural candidate for non-negative polynomials on $\bX(\vp)$ is the so-called quadratic module.

\begin{dfn}[Sum-of-squares and quadratic module]
    The set $\Sigma[\vx]$ of sum-of-squares (SoS) polynomials is defined by
    \begin{equation*}
        \Sigma[\vx]:= \left\{\sum_{i=1}^r q_i^2 \; : \; r \in \N, q_1,\ldots,q_r \in \R[\vx]\right\}.
    \end{equation*}
    For a vector of polynomials $\vp = (p_1,\ldots,p_m) \in \R[\vx]^m$, the quadratic module $\cQ(\vp)$ is defined by
    \begin{equation*}
        \cQ(\vp) := \left\{\begin{pmatrix} \vp^\top & 1\end{pmatrix}\bsigma \; : \; \bsigma = (\sigma_1,\ldots,\sigma_{m+1}) \in \Sigma[\vx]^{m+1}\right\}
    \end{equation*}
    and for $\ell \in \N$, the truncated quadratic module $\cQ_\ell(\vp)$ is
    \begin{equation}
        \cQ_\ell(\vp) := \left\{\begin{pmatrix} \vp^\top & 1 \end{pmatrix} \bsigma \in \cQ(\vp) \; : \; \deg(p_i \, \sigma_i) \leq 2 \ell \right\}.
    \end{equation}
\end{dfn}

It follows $g \geq 0$ on $\bX(\vp)$ for any $g \in \cQ(\vp)$. The celebrated Putinar's Positivstellensatz states that, under a certain compactness condition on $\bX$ (Archimedean condition), any strictly positive polynomial on $\bX$ belongs to $\cQ(\vp)$.

When $g$ is only non-negative on $\bX$, we cannot immediately apply Putinar's Positivstellensatz; but we can turn to $g + \varepsilon$ for $\varepsilon > 0$. This motivates the following questions:
\begin{center}
    For given $\varepsilon > 0$, what is the smallest $\ell = \ell(\varepsilon) \in \N$ such that $g+ \varepsilon \in \cQ_\ell(\vp)$?
\end{center}
The question if $\ell(0)$ exists, was (partially) answered positively in \cite{nie2014optimality}, where it was shown that generically (in $g$) it holds $\ell(0) < \infty$. However, whether for a given $g$ such finite convergence holds cannot be decided by a polynomial time algorithm unless $\mathrm{P} = \mathrm{NP}$, see \cite{vargas2024hardness}.

Since the existence of $\ell(0)$ is hard to answer, it is reasonable to investigate how $\ell(\varepsilon)$ behaves as $\varepsilon \rightarrow 0$. Bounding $\ell(\varepsilon)$ has strongly improved since \cite{nie2007complexity} to recent \cite{baldi2023effective} for generic semialgebraic sets $\bX$. Here, we use a specialized Positivstellensatz from \cite{slot2022sum} with even sharper bounds and consider the case $\bX = \bB := \{\vx \in \R^n: \|\vx\|_2 \leq 1\} 
 = \bX(b)$ for $b(\vx) = 1-\|\vx\|_2^2$. Following \cite{slot2022sum}, for $g\in \R[\vx]$ we set
\begin{equation}\label{eq:def:lb}
    \mathrm{lb}(g,\ell) := \max \{ \lambda \in \R: g-\lambda \in \cQ_\ell(b)\}
\end{equation}
and recall the following Positivstellensatz \cite{slot2022sum} on $\bB$.

\begin{thm}[{\cite[Theorem 3]{slot2022sum}}]\label{thm:LucasBall}
    Let $g\in \R[\vx]$ and $d = \deg(g)$. Consider $\bB = \bX(b)$ and set $g_{\mathrm{min}} := \min\limits_{\vx \in \bB} g(\vx)$, $g_{\mathrm{max}} := \max\limits_{\vx \in \bB} g(\vx)$. Then, for $\ell \geq nd$, it holds
    \begin{equation}
        g_{\mathrm{min}} - \mathrm{lb}(g,\ell) \leq \frac{c_n(d)}{\ell^2} \left( g_{\mathrm{max}} - g_{\mathrm{min}} \right)
    \end{equation}
    where $c_n \in \R[x]_n$ is a polynomial in one variable.
\end{thm}

We reformulate the above Positivstellensatz as a quantitative membership certificate for $\cQ_\ell(b)$.

\begin{cor}\label{cor:invertedLucasStatement}
    A sufficient condition for $g \in \R[\vx]$ positive on the unit ball $\bB$ to be in $\cQ_\ell(b)$ is that
\begin{equation} \label{eq:effectivePsatz}
    \ell^2 \geq c_n(d) \dfrac{g_{\max} - g_{\min}}{g_{\min}}
\end{equation}
\end{cor}
\begin{proof}
    Let $g \in \R[\vx]$ be positive on $\bB$. Then, $g \in \cQ_\ell(b)$ if and only if $\lambda=0$ is feasible in the optimization problem \eqref{eq:def:lb}, i.e. if and only if $\mathrm{lb}(g,\ell) \geq 0$. Using Theorem \ref{thm:LucasBall}, a sufficient condition for this to happen is that 
    $$g_{\min} \geq \dfrac{c_n(d)}{\ell^2}(g_{\max}-g_{\min})$$
    which can be rearranged into \eqref{eq:effectivePsatz}.
\end{proof}


\noindent

\subsection{A Positivstellensatz for trigonometric polynomials}
We call a function $f:[0,1]^n \rightarrow \mathbb{C}$ a trigonometric polynomial of bandwidth $d$, denoted by $f\in \cT_d$, if
\begin{equation*}\label{eq:TrigPol}
    f(\vx) = \sum\limits_{\bomega \in \lint -d,d\rint^n} f_{\bomega} e^{2\pi i \bomega^\top \vx}
\end{equation*}
for coefficients $f_{\bomega} \in \mathbb{C}$ for $\bomega \in \lint -d,d\rint^n$. By $\cT$ we denote the union of all sets $\cT_d$. For $f\in \cT$ let
\begin{equation}
    \|f\|_F := \sum\limits_{\bomega \in \lint -d,d\rint^n} |f_{\bomega}|.
\end{equation}
The truncated sum-of-squares of trigonometric polynomials are, for $\ell \in \N$,
\begin{equation*}
    \Sigma^\cT_\ell := \left\{ f \in \cT: f = \sum\limits_{i = 1}^s f_i^2, s \in \N, f_1,\ldots,f_s \in \cT_\ell\right\}.
\end{equation*}
As for $\cQ_\ell$, the set $\Sigma^\cT_\ell$ can be represented via positive semidefinite matrices, see \cite{lasserre2015introduction,bach2023exponential},
\begin{equation}\label{eq:SoSTrigRep}
    \Sigma^\cT_\ell = \{ f \in \cT: f = \bphi^{(\ell)*} \mQ \; \bphi^{(\ell)}, \text{ for some } \mQ \succeq 0 \}
\end{equation}
for $\bphi^{(\ell)} := (\phi^{(\ell)}_{\bomega})_{\bomega \in \lint -d,d \rint^n}$ given by
\begin{equation*}
    \phi^{(\ell)}_{\bomega}(\vx) := \frac{1}{(2\ell+1)^\frac{n-1}{2}} e^{-2\pi \bomega^\top \vx}.
\end{equation*}
Analog to (\ref{eq:def:lb}), for $q\in \cT$ and each $\ell\in \N$ we denote
\begin{equation*}
    \mathrm{lb}_\cT(q,\ell) := \max \{ c \in \R: q-c \in \Sigma^\cT_\ell \}.
\end{equation*}
Setting $q_\mathrm{min} := \min\limits_{\vx \in [0,1]^n} q(x)$ we recall a result from \cite{bach2023exponential}.

\begin{thm}[{\cite[Theorem 1]{bach2023exponential}}]\label{thm:BachRudi}
    Let $q \in \cT_{2d}$ be a real-valued. For $\ell \geq 3d$ it holds, for $q_0$ the mean value of $q$, that
    \begin{equation}\label{eq:ConvRateHypercubeTrig}
        q_\mathrm{min} - \mathrm{lb}_\cT(q,\ell) \leq \| q - q_0\|_F \left( \left( 1-\frac{6d^2}{\ell^2}\right)^{-d} - 1\right)
    \end{equation}
\end{thm}


\begin{cor}\label{cor:MembershipSoST}
    A sufficient condition for $q\in \cT_{2d}$ positive on $[0,1]^n$ to belong to $\Sigma^\cT_\ell$ is that
    \begin{equation*}
        \ell^2 \geq 12d^2n \max \left\{1,  \frac{\| q - q_0\|_F}{q_\mathrm{min}}\right\}
    \end{equation*}
\end{cor}

\begin{proof}
    The function $q$ belongs to $\Sigma_\ell^\cT$ if and only if $\mathrm{lb}_\cT(q,\ell) \geq 0$. For $\ell \geq 3d$, By Theorem \ref{thm:BachRudi}, this can be guaranteed when
    \begin{equation}\label{eq:PosCertTrig}
        q_\mathrm{min} \geq \| q - q_0\|_F \left( \left( 1-\frac{6d^2}{\ell^2}\right)^{-d} - 1\right).
    \end{equation}
    To bound the right-hand side, we use the Bernoulli inequality: $(1+x)^n \geq 1+nx$ for any $x \geq -1$. For $\ell^2 \geq 12d^2n$ we choose $x = -\frac{6d^2}{\ell^2} > -\frac{1}{2n}$ and get
    \begin{eqnarray*}
        \left( 1-\frac{6d^2}{\ell^2}\right)^{-n} - 1 & = & (1+x)^{-n} - 1 \leq \frac{1}{1 + nx} - 1\\
        & = & \frac{-nx}{1+nx} \leq \frac{-nx}{\frac{1}{2}} = \frac{12nd^2}{\ell^2}.
    \end{eqnarray*}
    We infer that (\ref{eq:PosCertTrig}) is satisfied for $\ell^2 \geq 12 d^2n \frac{\| q - q_0\|_F}{q_\mathrm{min}}$ (under the assumption $\ell^2 \geq 12d^2n$). This shows the statement.
\end{proof}

\section{Exit location estimation}

We consider the following setting from \cite{henrion2023moment}: Let a stochastic differential equation be given by
\begin{equation}\label{eq:StochasticODE}
   dX_t = \vf_0(X_t) \; dt + \mF(X_t) \; dW_t, \; \; X_0 = \vx_0
\end{equation}
for $\vf_0 = (f_{0i})_i :\R^n \rightarrow \R^n$, $\mF = (f_{ij})_{i,j} :\R^n \rightarrow \R^{n \times r}$, a deterministic initial condition $\vx_0$ and $(W_t)_{t \geq 0}$ a $r$-dimensional Brownian motion. For a constraint set $\bX \subset \R^n$ and a given function $g:\partial \bX \rightarrow \R$, we are interested in the expected exit value given by
\begin{equation}\label{eq:exitvalue}
    v^\star(\vx_0) := \mathbb{E}(g(X_{\tau}))
\end{equation}
where $\tau = \inf \{t\geq 0 \; ; \, X_t \in \partial\bX\}$ is the first time at which the process $(X_t)_t$ starting at $X_0 = \vx_0$ hits $\partial \bX$. In this text, we restrict to the case where $\bX = \bB$ is the unit ball. Furthermore, as in \cite{henrion2023moment,schlosser2024convergence}, we assume the following
\begin{asm}\label{cond:SDEDidier} \leavevmode
\begin{enumerate}

    \item[{\bf \ref*{cond:SDEDidier}.1}] $g,f_{0i},f_{ij} \in \R[\vx]$ for $i = 1,\ldots,n$ and $j = 1,\ldots,r$.
    \item[{\bf \ref*{cond:SDEDidier}.2}] $\mA(\vx) := \mF(\vx) \mF(\vx)^\top$ is positive definite for all $\vx\in \bB$.
\end{enumerate}
\end{asm}

\begin{rem}
    Under the above assumptions there exists a unique solution $X_t$ of (\ref{eq:StochasticODE}) for $t\leq \tau$, see \cite{henrion2023moment,oksendal2013stochastic}.
\end{rem}

The generator $\cL$ of the process $X_t$ is given by the second-order partial differential operator
\begin{equation*}
    \cL v (\vx) := -\sum\limits_{i,j = 1}^n a_{ij}(\vx) \frac{\partial^2 v}{\partial\vx_i\partial\vx_j}(\vx) + \sum\limits_{i = 1}^n f_{0i} (\vx) \frac{\partial v}{\partial \vx_i}(\vx)
\end{equation*}
for $(a_{ij}(\vx))_{i,j = 1,\ldots,n} = \mF(\vx) \mF(\vx)^\top$. We refer to \cite{oksendal2013stochastic,henrion2023moment} for details on the generator of a diffusion process $X_t$.

\begin{rem}\label{rem:vSolutionPDE}
    The function $v^\star$ from (\ref{eq:exitvalue}) is the solution of the elliptic second-order partial differential equation
    \begin{eqnarray}\label{eq:PDEv}
        \cL v & = & 0 \text{ on } \mathring{\bB}\\
        v & = & g \text{ on } \partial \bB =: \bS. \notag
    \end{eqnarray}
    where $\mathring{\bB}$ denotes the interior of the unit ball and $\bS = \partial \bB$ its boundary. Because $\bS$ is smooth and $g$ is smooth, it follows from elliptic regularity \cite{gilbarg1977elliptic,henrion2023moment} that a solution $v^\star \in C^\infty(\bB)$ of (\ref{eq:PDEv}) exists and is unique.
\end{rem}
The elliptic maximum principle, applied to (\ref{eq:PDEv}), yields that $v^\star(\vx_0)$ from (\ref{eq:exitvalue}) is the solution of the following infinite dimensional linear programming problem (LP), see \cite{henrion2023moment},
\begin{eqnarray}\label{eq:LPDidierSDE}
    v^\star(\vx_0)  = & \max\limits_v&  v(\vx_0)\\
    & \text{s.t.} & \cL v\leq 0 \text{ on } \; \bB\notag\\
    & & v\leq g \text{ on } \;\bS.\notag
\end{eqnarray}

Assumption \ref{cond:SDEDidier}.1 ensures that the LP (\ref{eq:LPDidierSDE}) is an instance of a GMP and thus the moment-SoS hierarchy framework can be applied, leading to a hierarchy of (convex) semidefinite programs (SDPs), see \cite{henrion2023moment}: For each $\ell \in \N$ consider
\begin{eqnarray}
        v^\star_\ell(\vx_0) := & \sup\limits_{v \in \R[\vx]} & v(\vx_0) \label{eq:HierarchyLPDidier} \\
         & s.t. & -\cL v \in \cQ_\ell(b) \notag \\
         & & g-v \in \cQ_\ell(-b,b) = \Sigma[\vx] + b\cdot\R[\vx] \notag
\end{eqnarray}
As a special case of \cite[Theorem 2]{henrion2023moment} it holds $v^\star_\ell(\vx_0) \rightarrow v^\star(\vx_0)$ as $\ell \rightarrow \infty$ and by \cite[Theorem 4.14]{schlosser2024convergence}, inserting $\bX = \bB$, we get the following convergence rate
\begin{equation}\label{eq:RateNaco}
    v^\star(\vx_0) - v^\star_\ell(\vx_0) \in \cO \left( \ell^{-\frac{1}{(2.5 + s)n}}\right) \text{ for any } s > 0.
\end{equation}

\section{A Positivstellensatz on the sphere}
With the use of specialized Positivstellensätze we want to improve the rate (\ref{eq:RateNaco}). Theorem \ref{thm:LucasBall} is such an effective Positivstellensatz for the unit ball $\bB$. For the sphere, we present and use the effective Positivstellensatz Corollary \ref{cor:schlosser-Psatz}, in which we transfer Theorem \ref{thm:BachRudi} to the sphere via spherical coordinates; more precisely, we use the map $\bpsi:[0,1]^{n-1}\rightarrow \bS$ from (\ref{eq:DefPsi}) in the Appendix.

\begin{cor} \label{cor:schlosser-Psatz}
    Let $p\in \R_d[\vx]$ and $q:= p \circ \bpsi$. Then we have
    \begin{equation*}
        p \geq 0 \text{ on } \bS \quad \text{ if and only if } \quad q \geq 0 \text{ on } [0,1]^{n-1}. 
    \end{equation*}
    Additionally, assume $p_{\mathrm{max}} \geq p \geq p_\mathrm{min} > 0$ on $\bS$. Then it holds $q \in \Sigma_\ell^\cT$ for
    \begin{equation*}
        \ell^2 \geq 48d^2 (n-1) (4d+1)^\frac{n}{2} \frac{p_{\mathrm{max}}}{p_\mathrm{min}}.
    \end{equation*}
\end{cor}

\begin{proof}
    Let $\psi$ be the map defined in the appendix. We will use the following of its properties: The map $\bpsi$ is a trigonometric polynomial of bandwidth $2$ and its image is $\bS$. From the first property, we infer that $q$ has bandwidth at most $2d$, because $p$ is of degree $d$. The first statement follows from the surjectivity of $\bpsi$ onto the sphere. For the second statement, note that  $p_{\mathrm{max}} \geq p \geq p_\mathrm{min} > 0$ implies
    \begin{equation}\label{eq:boundq}
        p_{\mathrm{max}} \geq q \geq p_\mathrm{min} \text{ on } [0,1]^{n-1}.
    \end{equation}
    By Corollary \ref{cor:MembershipSoST}, for $\ell^2 \geq 12(2d)^2n \max \left\{1,  \frac{\| q - q_0\|_F}{p_\mathrm{min}}\right\}$ we have $q\in \Sigma_\ell^\cT$; where we used that the bandwidth of $q$ is at most $2d$. It remains to show
    \begin{equation}\label{eq:Boundq-q0F}
        \| q - q_0\|_F \leq (4d+1)^\frac{n}{2} p_{\mathrm{max}}.    
    \end{equation}
    We write $q(x) = \sum\limits_{\omega \in \lint -2d,2d\rint} q_{\bomega} e^{i2\pi \bomega^\top x}$ and get by the Cauchy-Schwarz inequality and the Parseval equality that
    \begin{eqnarray*}
        \| q - q_0\|_F & = & (\sum\limits_{\omega} |q_\omega|) - |q_0| \leq \sum\limits_{\omega} |q_\omega|\\
        & \leq & (\sum\limits_{\omega} |q_\omega|^2)^\frac{1}{2} (4d+1)^\frac{n}{2}\\
        & = & \|q\|_{\mathrm{L}^2([0,1]^{n-1})} (4d+1)^\frac{n}{2}.
    \end{eqnarray*}
    Lastly, we get (\ref{eq:Boundq-q0F}) via $\|q\|_{\mathrm{L}^2([0,1]^{n-1})} \leq p_{\mathrm{max}}$ by (\ref{eq:boundq}) .
\end{proof}

\begin{rem}\label{rem:FangFawzi}
    The Positivstellensatz in \cite{fang2021quantum} also provides a quadratic bound in $\ell$; however, it is restricted to homogenous polynomials and therefore we cannot apply it to (\ref{eq:LPDidierSDE}).
\end{rem}

\section{SDP formulation and convergence rate}

Corollary \ref{cor:schlosser-Psatz} states that we can efficiently certify positivity of a polynomial $p \in \R[x]$ on $\bS$ via membership of $q:= p \circ \bpsi$ to $\Sigma_\ell^\cT$. Together with Corollary \ref{cor:invertedLucasStatement}, this motivates the following hierarchy of SDPs: For each $\ell \in \N$ consider
\begin{equation}\label{eq:Hierarchy+Trig}
    \begin{tabular}{ccc}
        $v^\star_{\cT\ell}(\vx_0) :=$ & $\sup\limits_{v \in \R[\vx]_{2\ell}}$ & $v(\vx_0)$\\
         & s.t. & $-\cL v \in \cQ_\ell(b)$\\
         & & $g \circ \bpsi - v\circ \bpsi \in \Sigma_\ell^\cT$.
    \end{tabular}
\end{equation}

\begin{rem}
    To show that (\ref{eq:Hierarchy+Trig}) can be stated as an SDP, note first that membership of a (trigonometric) polynomial to $\cQ_\ell(b)$ respectively $\Sigma_\ell^\cT$ is expressed by linear matrix inequalities, as in (\ref{eq:SoSTrigRep}). Because the maps $v\mapsto -\cL v$ and $v\mapsto g\circ \bpsi - v \circ \bpsi$ are affine, the feasible set for (\ref{eq:Hierarchy+Trig}) corresponds (linearly) to affine sections of cones of positive semidefinite matrices. Since the cost function in (\ref{eq:Hierarchy+Trig}) is linear in $v$, this shows that (\ref{eq:Hierarchy+Trig}) is an SDP for each $\ell \in \N$.
\end{rem}

\begin{rem}
    As $\ell \rightarrow \infty$, the overall number of decision variables has order $\ell^{2n}$, as for the SDP formulation of (\ref{eq:HierarchyLPDidier}).
\end{rem}

To obtain a convergence rate for $v^\star_{\cT\ell}$ to $v^\star$ as $\ell \rightarrow \infty$, we follow the strategy from \cite{schlosser2024convergence}, that is
\begin{enumerate}
    \item Approximate $v^\star$ by well-chosen feasible polynomials.
    \item Apply effective Positivstellensätze. 
\end{enumerate}
For 1), we let $u \in C^{\infty}(\overline{\bB})$ be the unique solution of the following Poisson problem
\begin{eqnarray}\label{eq:PoissionProbSol}
\begin{tabular}{cccll}
     $\cL u $ & $=$ & $-1$ & \text{in} & $\bB$\\ 
     $u$ & $=$ & $0$ & \text{on} & $\bS$.
\end{tabular}
\end{eqnarray}
The existence of a unique smooth solution $u$ is guaranteed by \cite{gilbarg1977elliptic}. For $\varepsilon > 0$ we define the function $v_\varepsilon$ by
\begin{equation}\label{eq:veps}
    v_\varepsilon := v^\star + \varepsilon u - \varepsilon
\end{equation}
We will see that the function $v_\varepsilon$ is strictly feasible and
for a suited choice of $\varepsilon$ and sufficiently close approximation of $v_\varepsilon$ by polynomials, the effective Positivstellensätze Corollary \ref{cor:invertedLucasStatement} and \ref{cor:schlosser-Psatz} lead to our following main result.

\begin{thm}\label{thm:ConvRate}
    Under Assumption \ref{cond:SDEDidier}, it holds
    \begin{equation*}
        0\leq v^\star(\vx_0) - v^\star_{\cT\ell}(\vx_0) \in \underset{\ell\to\infty}{\cO} \left( \ell^{-s}\right) \text{ for any } s\geq 0.
    \end{equation*}
\end{thm}

\begin{proof}
    We show first that $v^\star_{\cT\ell} (\vx_0)$ is a lower bound for $v^\star(\vx_0)$ for all $\ell \in \N$. Let $\ell \in \N$ and $v\in \R[\vx]$ be feasible for (\ref{eq:Hierarchy+Trig}). By feasibility of $v$ (and Corollary \ref{cor:schlosser-Psatz}) we have $-\cL v \geq 0$ on $\bB$ and $g - v\geq 0$ on $\bS$. For the function $\Tilde{v} := v^\star - v$ it follows, using (\ref{eq:PDEv}),
    \begin{eqnarray*}
        \cL \Tilde{v} & = & \cL v^\star - \cL v = -\cL v \geq 0 \text{ on } \bB\\   v & = &  g-v \geq 0 \text{ on } \bS.
    \end{eqnarray*}
    By Assumption \ref{cond:SDEDidier}.2 the operator $\cL$ is elliptic and the maximum principle for elliptic operators \cite[Theorem 3.1]{gilbarg1977elliptic} yields for all $x_0 \in \bB$ that $v^\star(\vx_0) - v(\vx_0) = \Tilde{v}(x_0) \geq 0$ on $\bB$. Because $v$ was an arbitrary feasible point for the optimization problem (\ref{eq:Hierarchy+Trig}), we conclude $0\leq v^\star(\vx_0) - v^\star_{\cT\ell}(\vx_0)$. To show the claimed convergence rate, let $v^\star,u \in C^\infty(\bB)$ be the solution of (\ref{eq:PDEv}) and (\ref{eq:PoissionProbSol}) respectively. Let $d,k \in \N$. By the Jackson-inequality \cite{bagby2002multivariate}, there exist $p_d,q_d \in \R_d[\vx]$ such that
    \begin{subequations}
     \begin{equation}\label{eq:Jackson}
        \|v^\star - p_d\|_{C^{2}(\bB)}, \|u - q_d\|_{C^{2}(\bB)} \leq \frac{\gamma_n(k)}{d^k}
    \end{equation}
    for some constant $\gamma_n(k)$ which depends only on $n$ and $k$ but \textit{not} on $d$. Additionally, we define the following constant
    \begin{equation}\label{eq:s2}
        C := \max \left\{4,\sup\limits_{\vx \in \bB} \sum\limits_{i,j = 1}^m |a_{ij}(\vx)| + \sum\limits_{i = 1}^m |f_{0i}(\vx)|\right\},    
    \end{equation}
    for which it holds for all $h \in C^2(\bB)$
    \begin{equation}\label{eq:LvleqAv}
        |\cL h| = \left|\sum\limits_{i,j = 1}^m a_{ij} \frac{\partial^2 h}{\partial\vx_i\partial\vx_j}+ \sum\limits_{i = 1}^m f_{0i} \frac{\partial h}{\partial \vx_i}\right|\leq C \|h\|_{C^2(\bB)}.
    \end{equation}
    The condition in (\ref{eq:s2}) of $C\geq 4$ is used later. For large enough $d\in \N$ (i.e. such that $\gamma_n(k)Cd^{-k} < 1$) we set
    \begin{equation}\label{eq:choiceeps}
        \varepsilon :=\frac{2\gamma_n(k)C}{1-\gamma_n(k)Cd^{-k}} d^{-k} \in \underset{d\to\infty}{\cO}(d^{-k}).
    \end{equation}
    \end{subequations}
    and, motivated by (\ref{eq:veps}), we define the candidate $v_d$ by
    \begin{subequations}
    \begin{equation}
        v_d := p_d + \varepsilon q_d - \varepsilon \in \R_d[\vx].
    \end{equation}
    We first show a degree bound on $\ell$ for the membership of $-\cL v_d$ to $\cQ_\ell(b)$ using Corollary \ref{cor:invertedLucasStatement}. By (\ref{eq:PDEv}) and (\ref{eq:PoissionProbSol}), we have
    \begin{eqnarray}\label{eq:Lvd}
        \cL v_d & = & \cL (v^\star + \varepsilon u) - \cL ((v^\star  -p_d) + \varepsilon (u-q_d))\notag \\
        & = & -\varepsilon - \cL (v^\star  -p_d) - \varepsilon \cL (u-q_d).
    \end{eqnarray}
    By (\ref{eq:LvleqAv}), it holds $\left|\cL (v^\star  -p_d)\right| \leq C \|v^\star  -p_d\|_{C^2(\bB)}$ and $\left|\cL (u-q_d)\right| \leq C \|u-q_d\|_{C^2(\bB)}$. Inserting this with (\ref{eq:Jackson}) into (\ref{eq:Lvd}) gives, with the choice of $\varepsilon$ in (\ref{eq:choiceeps}),
    \begin{align*}
        \cL v_d & \geq - C \frac{\gamma_n(k)}{d^k} \left(1+ 2\frac{1+\gamma_n(k)Cd^{-k}}{1-\gamma_n(k)Cd^{-k}} \right) \\
        \cL v_d & \leq - C \frac{\gamma_n(k)}{d^k}.
    \end{align*}
    For $d\in \N$ large enough, such that $\gamma_n(k)Cd^{-k} \leq \frac{1}{2}$, this simplifies to
    \begin{equation}\label{eq:minmaxLvdB}
        - 7 \frac{\gamma_n(k)C}{d^k} \leq \cL v_d \leq -\frac{\gamma_n(k)C}{d^k}.
    \end{equation}
    To bound the degree of $\cL v_d$ let $\deg \mA$ and $\deg \vf_0$ denote the maximum degree of the polynomials $a_{ij}$ for $i,j \in \lint 1,n\rint$ and $f_{0i}$ for $i \in \lint 1,n\rint$ respectively. We then have
    \begin{equation}\label{eq:DegreeLv}
        \deg (\cL v_d) \leq \underset{=:\hat{d} \in \cO(d)}{\underbrace{\max\{ d- 2+ \deg \mA, d-1+ \deg \vf_0 \}}}
    \end{equation}
    \end{subequations}
    Corollary \ref{cor:invertedLucasStatement} applied to $-\cL v_d$, with lower and upper bounds taken from (\ref{eq:minmaxLvdB}) and degree (\ref{eq:DegreeLv}), yields $-\cL v_d \in \cQ_\ell(b)$ for
    \begin{equation}\label{eq:boundellBall}
        \ell^2 \geq c_n(\hat{d}) \frac{7 \frac{\gamma_n(k)C}{d^k} - \frac{\gamma_n(k)C}{d^k}}{\frac{\gamma_n(k)C}{d^k}} = 6c_n(\hat{d})
    \end{equation}
    where $c_n$ is the univariate polynomial from Theorem~\ref{thm:LucasBall}. 
    Next, we treat the membership of $g-v_d$ to $\Sigma_\ell^\cT$. We want to apply Corollary \ref{cor:schlosser-Psatz} and therefore we first need to bound $g-v_d$ from above and below on $\bS$. We have, again by (\ref{eq:PDEv}) and (\ref{eq:PoissionProbSol}),
    \begin{eqnarray*}
        g-v_d & = & g- p_d - \varepsilon q_d + \varepsilon\\
        & = & g- v^\star - \varepsilon u + \varepsilon + (v^\star  -p_d) + (u  -q_d)\\
        & \geq & \varepsilon - \|v^\star  -p_d\|_{C^2} - \|u  -q_d\|_{C^2}\\
        & \overset{\text{(\ref{eq:Jackson})}}{\geq} & \varepsilon - \frac{2\gamma_n(k)}{d^k}.
    \end{eqnarray*}
    For $d$ large enough, such that $\gamma_n(k)Cd^{-k} > \frac{1}{2}$, we get by the choice of $\varepsilon$, see (\ref{eq:choiceeps}), and from $C \geq 4$, see (\ref{eq:s2}), that $\varepsilon \geq \frac{4\gamma_n(k)}{d^k}$. An analog computation for an upper bound of $g-v_d$ on $\bS$ then gives $\frac{2\gamma_n(k)}{d^k} \leq g-v_d \leq \frac{6\gamma_n(k)}{d^k}$. Now we can apply Corollary \ref{cor:schlosser-Psatz} to get $g\circ \bpsi - v_d \circ \bpsi \in \Sigma_\ell^\cT$ for
    \begin{equation*}
        \ell^2 \geq 48d^2 (n-1) (4d+1)^\frac{n}{2} \frac{\frac{6\gamma_n(k)}{d^k}}{\frac{2\gamma_n(k)}{d^k}} = 144d^2 (n-1) (4d+1)^\frac{n}{2}.
    \end{equation*}
    Together with (\ref{eq:DegreeLv}) we get that $v_d$ is feasible (for large enough $d$ such that $\gamma_n(k)Cd^{-k} \leq \frac{1}{2}$), for $\ell \in \N$ with
    \begin{equation}\label{eq:degl}
        \ell \geq \ell_d:=  \max \{ 6c_n(\hat{d}), 144d^2 (n-1) (4d+1)^\frac{n}{2}\} ^\frac{1}{2}.
    \end{equation}
    Finally, we have a look at how fast the cost of $v_d$ in (\ref{eq:Hierarchy+Trig}), namely $v_d(\vx_0)$, approaches the optimal cost $v^\star (\vx_0)$. We have
    \begin{eqnarray}\label{eq:costrate}
        v^\star(\vx_0) - v_d(\vx_0) & = & (v^\star - p_d + \varepsilon q_d - \varepsilon)(\vx_0)\notag\\
        & \leq & \|v^\star  -p_d\|_{C^2} + \varepsilon (q_d(\vx_0) + 1) \notag\\
        & \leq & \frac{\gamma_n(k)}{d^{k}} + \varepsilon \left(u(\vx_0) + \frac{\gamma_n(k)}{d^k} +1\right) \notag \\
        &\in& \underset{d\to\infty}{\cO}(d^{-k})
    \end{eqnarray}
    since $\varepsilon \in \cO(d^{-k})$ by (\ref{eq:choiceeps}). To conclude the statement, for $\ell \in \N$ large enough we need an upper bound $d_\ell$ on $d$ such that $\ell_d \leq \ell$. 
    Let us first analyze $c_n$: By~\cite[eq. (39)]{slot2022sum} we get
    \begin{subequations}
        \begin{equation} \label{eq:39}
            c_n(d) = 2(n+1)^2 d^2 \Gamma_n(\bB)_d
        \end{equation}
        where~\cite[Section 4.1]{slot2022sum} gives
        \begin{eqnarray} \label{eq:p15}
            \Gamma_n(\bB)_d^2 & \leq & \max\limits_{0\leq k \leq d} \left(1+\frac{2k}{n-1}\right) \binom{k+n-2}{k} \notag \\
            & \leq & \left(1+\frac{2d}{n-1}\right) \binom{d+n-2}{d} \notag \\
            & \leq & \left(1+\frac{2d}{n-1}\right) (d+1)^{n-2}.
        \end{eqnarray}
        Hence, combining~\eqref{eq:39} and~\eqref{eq:p15} yields
        \begin{equation*}
            c_n(d) \leq 2(n+1)^2d^2\left(1+\frac{2d}{n-1}\right)^{\nicefrac{1}{2}}(d+1)^{\nicefrac{n}{2} -1}
            \end{equation*}
        and with~\eqref{eq:DegreeLv} we deduce the following asymptotic behavior:
        \begin{equation} \label{eq:cn(d)}
            c_n(\hat{d}) \in \underset{d\to\infty}{\cO}\left(d^{\nicefrac{(n+3)}{2}}\right)
        \end{equation}
    \end{subequations}
    For the second term in the expression in~\eqref{eq:degl}, we have $144d^2(n-1)(4d+1)^{\nicefrac{n}{2}} \in \underset{d\to\infty}{\cO}\left(d^{\nicefrac{(n+4)}{2}}\right)$ which we combine with~\eqref{eq:degl} and~\eqref{eq:cn(d)} to get
    $$ \ell_d \in \underset{d\to\infty}{\cO}\left(d^{\nicefrac{(n+4)}{4}}\right) $$
    By inverting this asymptotic estimate of $\ell_d$, it holds
    \begin{equation}\label{eq:dell}
        d_\ell \in \underset{\ell\to\infty}{\cO}\left(\ell^{\frac{4}{n+4}}\right)
    \end{equation}
    and, for $\ell$ large enough such that $\varepsilon > 0$ in (\ref{eq:choiceeps}) and $\gamma_n(k)Cd_\ell^k \leq \frac{1}{2}$ (this was used in (\ref{eq:minmaxLvdB})), the polynomial $v_{d_\ell}$ is feasible for (\ref{eq:Hierarchy+Trig}) for $\ell$. Formulating (\ref{eq:costrate}) in terms of $\ell$ gives, using the estimate \eqref{eq:dell}, the convergence rate
    \begin{equation}
        v^\star(\vx_0) - v_{d_\ell}(\vx_0) \in \cO\left(\ell^{-\frac{4k}{n+4}}\right).
    \end{equation}
    Because $k\in \N$ was arbitrary the statement follows.
\end{proof}


\begin{rem}
    The convergence rate in Theorem \ref{thm:ConvRate} is not independent of the dimension $n$. In the proof, for each $s\geq 0$, we find constants $M(n,s)$ such that $v^\star(\vx_0) - v^\star_{\cT\ell}(\vx_0) \leq M(n,s) \ell^{-s}$ holds asymptotically as $\ell \rightarrow \infty$. However, the constants $M(n,s)$ are not uniformly bounded in $(n,s)$.
\end{rem}

\begin{rem}\label{rem:exploitingSmoothnessAndConstraint}
    The reason why we overcome the polynomial convergence rates from \cite{schlosser2024convergence} for the problem (\ref{eq:HierarchyLPDidier}) is threefold. First, the optimal solution $v^\star$ of (\ref{eq:LPDidierSDE}) is smooth (we used this most efficiently in (\ref{eq:costrate})). Second, we use specialized Positivstellensätze with polynomial scaling in the degree $d$. Third, the solution $v^\star$ satisfies the inequality constraints in (\ref{eq:LPDidierSDE}) with equality. This affected the terms  $\frac{p_\mathrm{max} - p_\mathrm{min}}{p_\mathrm{min}}$ and $\frac{p_\mathrm{max}}{p_\mathrm{min}}$ when we applied Corollary \ref{cor:invertedLucasStatement} and \ref{cor:schlosser-Psatz}. Namely, for the candidate functions $v_{d_\ell}$ the corresponding terms are in $\cO(1)$ as $\ell \rightarrow \infty$ even though the denominator tended to zero.
\end{rem}

\section{Conclusion}
    We continue the pathway of recent research and interest in convergence rates in the moment-SoS hierarchy for generalized moment problems. We focus on the example from \cite{henrion2023moment} of computing exit location estimations for stochastic processes, and, by restricting to the unit ball, we improve our recent polynomial rate in \cite{schlosser2024convergence} to a super-polynomial convergence rate. Such a convergence rate is remarkable for the moment-SoS hierarchy because it lies in between the recent polynomial convergence rates for static polynomial optimization on the hypercube respectively the sphere \cite{baldi2024degree,slot2022sum,fang2021quantum,baldi2023effective,bach2023exponential}, and the exponential rate in \cite{bach2023exponential} which is obtained under additional regularity assumptions.

    The main ingredients for our improvement compared to recent work \cite{schlosser2024convergence} are: First, we restrict to a problem setting that allows us to apply specialized Positivstellensätze. More precisely, we apply a Positivstellensatz on the unit ball from \cite{slot2022sum} (see Theorem \ref{thm:LucasBall}) and we present a new Positivstellensatz on the sphere, see Corollary \ref{cor:schlosser-Psatz}. Secondly, we exploit two strong properties of the problem at hand and their interaction with the chosen Positivstellensätze, see also Remark \ref{rem:exploitingSmoothnessAndConstraint}.
    
    Consequently, the convergence rate from Theorem \ref{thm:ConvRate} does not need to transfer to moment-SoS hierarchies for other problems. However it highlights the importance of adapted and specialized Positivstellensätze and the exploitation of certain problem intrinsic properties in the moment-SoS hierarchy.
    
    Possible continuations include applying the same line of reasoning to different problems with GMP formulations or obtaining explicit bounds on the optimal value, here $v^\star$, based on computed optimal values in the hierarchy, here $v^\star_{\cT,\ell}$.


\section*{Appendix}
Let $\hat{\bpsi} = (\hat{\psi}_1,\ldots,\hat{\psi}_n):\R^{n-1} \rightarrow \R^n$ be defined as $\hat{\bpsi}_i(\theta_1,\ldots,\theta_{n-1}) = \cos(\pi\theta_i)\prod\limits_{j = 1}^{i-1} \sin(\pi\theta_j)$ for $1\leq i\leq n-2$ and $\hat{\bpsi}_{n-1}(\theta_1,\ldots,\theta_{n-1}) = \prod\limits_{j = 1}^{n-2} \sin(\pi\theta_j) \cos(2\pi \theta_{n-1})$ and $\hat{\bpsi}_n(\theta_1,\ldots,\theta_{n-1}) = \prod\limits_{j = 1}^{n-2} \sin(\pi\theta_j) \sin(2\pi \theta_{n-1})$. The map $\hat{\bpsi}$ restricted to $[0,1]^{n-2}\times [0,1)$ is called spherical coordinates because it is injective and it holds, see for instance \cite{blumenson1960derivation},
\begin{equation}\label{eq:ImageHatPhi}
    \hat{\phi}\left( [0,1]^{n-2}\times [0,1)\right) = \bS.
\end{equation}
The function $\hat{\bpsi}$ is 2-periodic in each variable; therefore we consider
\begin{equation}\label{eq:DefPsi}
    \bpsi (\theta) := \hat{\bpsi}(2\theta).    
\end{equation}
for which each component is a trigonometric polynomial of bandwidth $2$ (because the highest frequency appears in the last two components $\psi_{n-1}$ and $\psi_n$ and is $2$). Further, it holds
\begin{equation}\label{eq:ImagePsi}
    \bpsi([0,1]^{n-1}) = \bS.    
\end{equation}
This follows from (\ref{eq:ImageHatPhi}), i.e. $\bpsi([0,\frac{1}{2}]^{n-1}) = \bS$, and the identities $\sin (2\pi (\frac{1}{2} + \vx)) = -\sin(2\pi \vx)$, $\cos (2\pi (\frac{1}{2} + \vx)) = -\cos(2\pi \vx)$ and $\sin(4\pi (\frac{1}{2} + \vx)) = \sin(4\pi \vx)$, $\cos(4\pi (\frac{1}{2} + \vx)) = \cos(4\pi \vx)$ applied to each of the coordinates of $\bpsi$.

\section*{Acknowledgements}

The authors are grateful to Didier Henrion, Francis Bach, and Alessandro Rudi for their feedback on this work.

\bibliographystyle{plain} 
\bibliography{references}

\end{document}